\documentclass[11pt]{amsart}
\usepackage{amscd,amsmath,amssymb,amsfonts}
\usepackage{url,enumerate,color,hyperref}
\usepackage{dsfont}
\usepackage[cmtip, all]{xy}

\usepackage[top=30mm,right=30mm,bottom=30mm,left=20mm]{geometry}

\theoremstyle{plain}
\newtheorem{thm}{Theorem}
\newtheorem{lem}[thm]{Lemma}
\newtheorem{cor}[thm]{Corollary}

\theoremstyle{definition}

\newtheorem{rmk}[thm]{Remark}

\numberwithin{equation}{thm}

\newcommand{\Tr}{{\rm Tr}}

\newcommand{\Stab}{{\rm Stab}}
\newcommand{\Ker}{{\rm Ker}}

\newcommand{\alt}{\mathsf{A}}

\newcommand{\Irr}{\mathrm{Irr}}

\newcommand{\eps}{\epsilon}

\newcommand{\Ind}{\text{Ind}}
\newcommand{\diag}{\text{diag}}

\newcommand{\SL}{\mathrm{SL}}
\newcommand{\PSL}{\mathrm{PSL}}
\newcommand{\GL}{\mathrm{GL}}

\newcommand{\PSU}{\mathrm{PSU}}

\newcommand{\SO}{\mathrm{SO}}
\newcommand{\GO}{\mathrm{O}}

\newcommand{\sO}{{\mathcal O}}

\newcommand{\C}{{\mathbb C}}
{

\newcommand{\F}{{\mathbb F}}

\newcommand{\Z}{{\mathbb Z}}

\newcommand{\0}{\emptyset}

\newcommand{\ZB}{{\mathbf Z}}
\newcommand{\CB}{{\mathbf C}}

\newcommand{\bz}{\boldsymbol{z}}

\begin{document}
\title{The conjugation representation of the binary modular congruence group}
\author{Pham Huu Tiep}
\address{Department of Mathematics, Rutgers University, Piscataway, NJ 08854}
\email{tiep@math.rutgers.edu}
\thanks{The author gratefully acknowledges the support of the NSF (grant
DMS-2200850), the Simons Foundation, and the Joshua Barlaz Chair in Mathematics.}
\thanks{Part of this work was done while the author was visiting Princeton University and MIT. 
It is a pleasure to thank Princeton University and MIT for generous hospitality and stimulating environment.} 
\thanks{The author thanks Richard Hain for bringing the problem to his attention, and Gabriel Navarro and Eamonn O'Brien
for helpful conversations on the problem and for several computer calculations in the cases $p=2,3$.}  

\begin{abstract}
Motivated by the study of an Hecke action on iterated Shimura integrals undertaken in 
\cite{Hain}, in this appendix to \cite{Hain} we prove that, for any
prime $p \geq 5$ and for any integer $n \geq 1$, every complex irreducible representation of $G=\SL_2(\Z/p^n \Z)$ that are
 trivial on $\ZB(G)$ appears as an irreducible constituent of  
the conjugation representation of $G$. 
\end{abstract}

\maketitle

Let $G$ be a finite group. Consider the action of $G$ on (the vector space of) the complex group algebra 
$\C G = \langle x \mid x \in G \rangle_{\C}$, via its
conjugation on the natural basis $\{x \mid x \in G\}$: 
$$g: x \mapsto gxg^{-1}.$$
The corresponding representation $\Pi=\Pi_G$ is called {\it the conjugation representation}, or {\it the adjoint representation}.
This representation and its character $\pi_G$ are standard objects
of study in the representation theory of finite groups. For instance, the
submodule of $G$-invariants of this module forms the center of the
group algebra $\C G$. The multiplicity of the trivial representation $1_G$ in 
$\Pi_G$ equals the number $k(G)$ of conjugacy classes in $G$. However, almost nothing is known about the
multiplicities of other irreducible summands in $\Pi_G$. The problem has a
long history; and most of study was to determine these irreducible 
constituents of $\Pi_G$. See \cite{R}, \cite{Sch}, \cite{P} and
the bibliography there.

It is clear that any irreducible constituent of $\Pi_G$ is trivial on the center $\ZB(G)$. This condition turns out to be {\it not} sufficient.
More generally, the problem of determining irreducible constituents of $\Pi_G$ for finite simple groups has been completely 
solved, for alternating and sporadic groups in \cite{HZ} and for groups of Lie type in \cite{HSTZ}:

\begin{thm}{\rm \cite[Theorem 1.1]{HSTZ}}\label{main-hstz}
Let $G$ be a finite non-abelian simple group, other than 
$\PSU_n(q)$ with $n \geq 3$ coprime to $2(q+1)$. Then every irreducible 
character of $G$ is a constituent of the conjugation character $\pi_G$.
In the exceptional case, every irreducible character occurs, except
precisely the (unique)   
irreducible character of dimension $(q^n-q)/(q+1)$.
\end{thm}
 
As a special case of Theorem \ref{main-hstz}, we see that $\pi_G$ contains $\Irr(G)$ (the set of complex irreducible characters of 
$G$) when $G = \PSL_2(p)$ for any prime
$p \geq 5$. Equivalently, the irreducible constituents of $\pi_G$, where $G = \SL_2(p)$, are precisely the irreducible characters of
$G$ that are trivial on $\ZB(G)$. 

As shown in \cite{Hain}, it is of importance for the study of certain Hecke action on iterated Shimura integrals, see \cite[\S23]{Hain},
to extend the above result to all binary modular groups $\SL_2(\Z/p^n\Z)$, for all primes $p$ and all integers $n \geq 1$.

The main result of this paper settles this question for $p \geq 5$.

\begin{thm}\label{main-sl2}
Let $p \geq 5$ be any prime, $n$ any natural number, and let $G = \SL_2(\Z/p^n\Z)$. Then every irreducible character of
$G$ that is trivial on the central involution $\bz=-I_2$ of $G$ is a constituent of the conjugation character $\pi_G$.
\end{thm} 

First we recall some well-known facts.

\begin{lem}\label{red}
Let $G$ be a finite group and let $\chi \in \Irr(G)$. Then $\chi$ is a constituent of $\pi_G$ if and only if there is some 
$g \in G$ such that the restriction of $\chi$ to the centralizer $\CB_G(g)$ contains the principal character $1_{\CB_G(g)}$.
\end{lem}

\begin{proof}
It is clear that $\pi_G$ is the permutation character of the action of $G$ on itself via conjugation. The orbits of this action 
are precisely the conjugacy classes $g_i^G$, $1 \leq i \leq k(G)$. Hence $\pi_G = \sum^{k(G)}_{i=1}\pi_i$, where 
$\pi_i$ is the permutation character of $G$ acting on $g_i^G$. Now, $\pi_i$ is just the induced character 
$\Ind^G_{\CB_G(g_i)}(1_{\CB_G(g_i)})$. By Frobenius' reciprocity, 
$$[\pi_i,\chi]_G = [1_{\CB_G(g_i)},\chi|_{\CB_G(g_i)}]_{\CB_G(g_i)},$$
and hence $\chi$ is an irreducible constituent of $\pi_i$ if and only if $\chi|_{\CB_G(g_i)}$ contains the principal character 
$1_{\CB_G(g_i)}$. 
\end{proof}

\begin{lem}\label{quotient}
Let $G$ be a finite group and let $N \lhd G$ be any normal subgroup. Suppose $\theta \in \Irr(G/N)$ is a constituent of $\pi_{G/N}$.
Then $\theta$, viewed as a character of $G$ by inflation, is also a constituent of $\pi_G$.
\end{lem}

\begin{proof}
By Lemma \ref{red}, there is some $g \in G$ such that $\theta|_{\bar C}$ contains $1_{\bar C}$, where 
$\bar C := \CB_{G/N}(\bar g)$ for $\bar g:=gN$. We can write $\bar C=D/N$ for some subgroup $D$ of $G$ that contains 
$N$. Inflating $\theta$ to a $G$-character and $1_{\bar C}$ to a $D$-character, we then have that $\theta|_D$ contains $1_D$.  
But $D$ contains $1_{\CB_G(g)}$, so $\theta|_{\CB_G(g)}$ contains $1_{\CB_G(g)}$. 
\end{proof}

\begin{cor}\label{psl2}
If $p \geq 2$ is a prime and $G = \SL_2(p)$, then $\pi_G$ contains every irreducible character of $G$ that is trivial on 
$\ZB(G)$.
\end{cor}

\begin{proof}
By Theorem \ref{main-hstz}, $\pi_{G/Z}$ contains $\Irr(G/Z)$ if $p \geq 5$.
The same holds for $p=3$. (Indeed, $\PSL_2(3) \cong \alt_4$ has four irreducible characters, three of degree $1$ and one of degree 
$3$. The first three are trivial at the centralizer of any involution in $\PSL_2(3)$, and the one of degree $3$ restricts to the regular 
character of the centralizer of any element of order $3$. Hence the claim follows from Lemma \ref{red}.)
Now the statement follows from Lemma \ref{quotient} for $p\geq 3$.

For $p=2$, $\SL_2(2) \cong \mathsf{S}_3$ has $3$ irreducible characters. Two of them have degree $1$ and restrict to the 
principal character of the centralizer of a $3$-cycle. The remaining one, of degree $2$, restricts to the regular character of 
the centralizer of a $2$-cycle. Hence the statement follows from Lemma \ref{red}.
\end{proof}

\begin{lem}\label{f3}
Let $p \geq 3$ be any odd prime, $q=p^a$ with $a \in \Z_{\geq 1}$, and consider $V:=\F_q^3$ 
with a non-degenerate symmetric bilinear form $(\cdot,\cdot)$
and the corresponding special orthogonal group $\SO(V)$.
Let $U$ be a $2$-dimensional non-degenerate subspace of $V$ that contains a nonzero vector $u$ with $(u,u)=0$. If $0 \neq w \in V$, 
then the orbit of $w$ with respect to $\Omega(V)=[\SO(V),\SO(V)]$ intersects $U$. 
\end{lem}

\begin{proof}
Since $U$ contains $u$ and is non-degenerate, we can find a basis $\{u,v\}$ of $U$ such that $(u,u)=(v,v)=0$ and 
$(u,v)=1$ (see e.g. the proof of \cite[Proposition 2.5.3]{KlL}). For any $\lambda \in \F_q^\times$, we have
$(u+\lambda v,u+\lambda v) = 2\lambda$. It follows that there is some $0 \neq w_0 \in U$ such that 
$(w_0,w_0)=(w,w)$. By Witt's lemma (see e.g. \cite[Proposition 2.1.6]{KlL}), there is some $g_0 \in \GO(V)$ such that 
$g_0(w) = w_0$. Note that $G:=\GO(V) = \langle -I_3 \rangle \times \SO(V)$, and $\Omega(V)$ has index $2$ in $\SO(V)$, 
see \cite[Table 2.1.C]{KlL}.

Assume in addition that $(w,w) \neq 0$. Then for $W:=w^\perp$ we have 
$\Stab_G(w) = \GO(W) \cong \GO^\pm_2(q)$, and $[\GO(W):\SO(W)]=[\SO(W):\Omega(W)]=2$, see 
\cite[Proposition 2.9.1(iii)]{KlL}. It follows that $G=\Omega(V)\Stab_G(w)$, and so we can write $g_0=gh$ with 
$g \in \Omega(V)$ and $h(w)=w$. Now we have $g(w)=gh(w)=g_0(w)=w_0 \in U$, as desired.

Now assume that $(w,w)=0$. The above argument shows that $G$ acts transitively on the set $\sO$ of isotropic vectors in $V$.  
It is straightforward to check that 
$$|\Stab_G(w)|=2|\Stab_{\SO(V)}(w)| = 2q.$$
Hence $|\sO| = q^2-1$, and $\SO(V)$ acts transitively on $\sO$. 
As $\Stab_{\SO(V)}(w)$ has odd order $q$, it is contained in 
$\Omega(V)$, and so $\sO$ splits into two $\Omega(V)$-orbits, each of length $(q^2-1)/2$. 
Since $w,w_0,u \in \sO$, we may assume that $g \in \SO(V)$ and $w_0=u$. If $g \in \Omega(V)$, we are done. 
If not, then we can find a non-square element $\delta \in \F_q^\times$, and note that the element 
$s:=\diag(\delta,\delta^{-1})$ of $\SO(U)$ is not contained in $\Omega(V)$ (since $\Omega(U) \cong C_{(q-1)/2}$ and 
$\SO(U) \cong C_{q-1}$). Now $sg(w) = s(u) = \delta u \in U$, and we are again done. 
\end{proof}

\begin{lem}\label{p-com1}
Let $E$ be a finite group, $F \leq \ZB(E)$, $x \in E$, and let $\lambda$ be an irreducible character of $F$. Suppose we can find $h \in E$ such that the commutator $[x,h]$ belongs to $F \smallsetminus \Ker(\lambda)$. Then $\varphi(x)=0$ for any irreducible character
$\varphi$ of $E$ that lies above $\lambda$. 
\end{lem}

\begin{proof}
Let $\varphi$ be afforded by a representation $\Phi:E \to \GL_d(\C)$. By hypothesis, $\Phi([x,h]) = \alpha I_d$ for some 
$1 \neq \alpha \in \C^\times$. Now we have $\alpha I_d = \Phi(x)\Phi(hx^{-1}h^{-1})$, and so
$$\Phi(x) = \alpha \Phi(hxh^{-1}).$$
Taking trace, we obtain $\varphi(x) = \alpha \varphi(hxh^{-1}) = \alpha \varphi(x)$, and so $\varphi(x)=0$.
\end{proof}

Now we fix any prime $p \geq 3$, any integer $n \geq 2$, and consider
$$G = G_n := \SL_2(\Z/p^n\Z).$$
The reduction modulo $p$ induces the exact sequence 
\begin{equation}\label{ses1}
  1 \to E \to G \to G_1= \SL_2(\Z/p\Z) \to 1,
\end{equation}  
whereas the reduction modulo $p^{n-1}$ induces the exact sequence
\begin{equation}\label{ses2}
  1 \to F \to G \to G_{n-1} = \SL_2(\Z/p^{n-1}\Z) \to 1.
\end{equation}  
We will represent elements in $G$ by $2 \times $2-matrices over $\Z$. In particular, any element in $E$ can be written as 
$I_2+pX$ for some $X \in M_2(\Z)$, whereas
\begin{equation}\label{for-f}
  F=\{ I_2+p^{n-1}Y \mid Y \in M_2(\Z),\Tr(Y) = 0\},
\end{equation}  
where the trace condition is equivalent to having determinant $1$.

For any $i,j \in \Z_{\geq 1}$ with $i+j \geq n$ and any $X,Y \in M_2(\Z)$ we have   
\begin{equation}\label{com10}
  (I_2+p^iX)(I_2+p^jY) = I_2+p^iX+p^jY = (I_2+p^jY)(I_2+p^iX)
\end{equation}
in $G$. It follows that 
$$F \leq \ZB(E),$$
and so the kernel of the conjugation action of $G$ on $F$ contains $E$ and $\bz=-I_2$. Thus $F$ can be viewed as the $\F_pS$-module
$$V = \{ Z \mid Z \in M_2(\F_p),\Tr(Z) = 0\}$$ 
of size $p^3$, where the ``linearization'' is given by the map 
\begin{equation}\label{com11}
  I_2 +p^{n-1}Y \mapsto Y\mathrm{(mod}~p)
\end{equation}  
using \eqref{for-f}, and $S:=G_1/\langle \bz \rangle \cong \PSL_2(p)$. 

Under the identification \eqref{com11}, the action of $G_1$ on $V$ is via $g: Z \mapsto gZg^{-1}$. (In fact, one can show that
$V$ is the adjoint module for $G_1$.)
Note that the element $\begin{pmatrix}0 & 1\\-1 & 0\end{pmatrix}$ of order $2$ of $S$ acts nontrivially on $V$. Now, if $p > 3$, then $S$ is simple, so $S$ must act faithfully on $V$. If $p=3$,
then any nontrivial normal subgroup of $S \cong \alt_4$ contains any element of order $2$ of $S$, whence $S$ also acts faithfully on $V$ 
as well. 

The action of $S$ on $V$ preserves the trace form
\begin{equation}\label{tr1}
  (Z,T) = \Tr(ZT)
\end{equation}  
which yields a non-degenerate symmetric bilinear form on $V$. As $S$ has order $p(p^2-1)/2 = |\Omega(V)|$ and $S$ acts 
faithfully on $V$, we have therefore shown that
$$\Omega(V) = S.$$

Note that the element $1-p$ has order $p^{n-1}$ in $(\Z/p^n\Z)^\times$. Hence we can find a generator $t(1-p)$ of 
$(\Z/p^n\Z)^\times$, where $t$ has order $p-1$. We now fix the elements
\begin{equation}\label{com12}
  A:= \begin{pmatrix}1-p & 0\\0 & 1+p+p^2 + \ldots + p^{n-1} \end{pmatrix},~
  g:= \begin{pmatrix}t(1-p) & 0\\0 & t^{p-2}(1+p+p^2 + \ldots + p^{n-1}) \end{pmatrix}
\end{equation}
of $G=G_n$, and note that $A \in E$ has order $p^{n-1}$ and $g$ has order $p^{n-1}(p-1)$. 

\begin{lem}\label{cent}
Suppose $p \geq 5$. For the element $g$ defined in \eqref{com12} of $G=G_n$, we have
$$C:=\CB_G(g) = \langle g \rangle.$$
\end{lem}

\begin{proof}
We will proceed by induction on $n$, for which the induction base $n=1$ is clear. For induction step $n \geq 2$, 
it suffices to show that any element $h \in \CB_G(g)$ is contained in $\langle g \rangle$. Using the 
sequence \eqref{ses2} and the induction hypothesis, we see that $\CB_{G/F}(g) = \langle gF \rangle$. 
As $\CB_{G/F}(g) = \langle gF\rangle$, we can adjust $h$ by some power of $g$ and assume that 
$$h \in \CB_F(g).$$
In the conjugation action of 
$G/E \cong \SL_2(p)$ (see \eqref{ses1}) on $F$, $g$ acts via $\diag(t,t^{-1})$, and so has a $1$-dimensional fixed point subspace,
generated by $\diag(1,-1)$. It follows that $|\CB_F(g)|=p$, whence $\CB_F(g) = \langle A^{p^{n-2}} \rangle$.
Thus $h \in \langle A^{p^{n-2}} \rangle$, and so $h \in \langle g \rangle$. 
\end{proof}

Also, let $F_0$ be the subgroup of $F$ that corresponds to the subspace
$$V_0:= \left\{ \begin{pmatrix} 0 & b\\c & 0 \end{pmatrix} \mid b,c \in \F_p \right\}$$
of $V$ under the identification \eqref{com11}.

\begin{lem}\label{p-com2}
Assume that $n \geq 3$, $0 \leq j \leq n-3$, and $i$ any integer coprime to $p$. Then
$$\left\{ [A^{ip^j},B] \mid B=I_2+p^{n-j-2}Y  \in E,Y \in M_2(\Z) \right\} = F_0.$$ 
\end{lem}

\begin{proof}
Set $X:= \diag(-i,i)$. An induction on $j \geq 0$ shows that 
$$A^{ip^j} = I_2 +p^{j+1}(X+pX_1),~~A^{-ip^j} = I_2 -p^{j+1}(X+pX_2)$$
for some $X_1,X_2 \in M_2(\Z)$ depending on $j$. For any $B = I_2+p^{n-j-2}Y$, modulo $p^n$ we have
$$\begin{aligned}A^{ip^j}B & \equiv A^{ip^j}+p^{n-j-2}Y+p^{n-1}XY,\\
A^{ip^j}BA^{-ip^j} & \equiv I_2+p^{n-j-2}Y-p^{n-1}YX+p^{n-1}XY\\
                           & =B+p^{n-1}(XY-YX),\\
A^{ip^j}BA^{-ip^j}B^{-1} & \equiv I_2 + p^{n-1}(XY-YX),
\end{aligned}$$  
(for the last congruence we also used $B^{-1} \equiv I_2 \pmod{p}$). Writing $Y = \begin{pmatrix}a & b\\c &d \end{pmatrix}$
with $a,b,c,d \in \Z$ we have $XY-YX = \begin{pmatrix}0 & -2b\\2c &0 \end{pmatrix}$, showing 
$[A^{ip^j},B] \in F_0$. Since $F_0 \leq \ZB(E)$, for any $B' \in E$ we have 
$$[A^{ip^j},B'][A^{ip^j},B] = A^{ip^j}B'A^{-ip^j} \cdot A^{ip^j}BA^{-ip^j}B^{-1} \cdot B'^{-1} = [A^{ip^j},B'B],$$
which shows that 
$$F':= \left\{ [A^{ip^j},B] \mid B=I_2+p^{n-j-2}Y  \in E,Y \in M_2(\Z) \right\}$$ 
is a subgroup of $F_0$. 
Taking $a=d=c=0$ and $0 \leq b \leq p-1$, respectively $a=d=b=0$ and $0 \leq c \leq p-1$, we see that $F'$ contains all the elements 
of the form $\begin{pmatrix}0 & -2b\\0 & 0 \end{pmatrix}$ and of the form  
$\begin{pmatrix}0 & 0\\2c & 0 \end{pmatrix}$. Hence the subgroup $F'$ coincides with $F_0$.
\end{proof}

\begin{proof}[Proof of Theorem \ref{main-sl2}.]
We proceed by induction on $n \geq 1$, where the induction base $n=1$ is noted in Corollary \ref{psl2}.

For the induction base $n \geq 2$, we consider the normal subgroup $F \lhd G=G_n$ in \eqref{ses2}, which is an 
elementary abelian subgroup of order $p^3$, and any irreducible character $\chi$ of $G$ which is trivial at the central involution $\bz$.
If $\Ker(\chi) \geq F$, then the induction hypothesis implies by Lemma \ref{quotient} that $\chi$ is a constituent of $\pi_G$. So 
we will assume that $\Ker(\chi) \not\geq F$. Hence $\chi|_F$ contains a nontrivial irreducible character $\lambda$ of $F$, and 
by the Clifford correspondence,
$$\chi = \Ind^G_J(\psi),$$
where $J = I_G(\lambda)$ is the inertia group of $\lambda$ (so $J \rhd F$), and $\psi \in \Irr(J)$ lies above $\lambda$. In fact
we have $J \geq E$ since $F \leq \ZB(E)$.

By Lemma \ref{red}, it suffices to show that $\chi|_C$ contains $1_C$ for the centralizer $C=\langle g \rangle$ of the element
$g$ in Lemma \ref{cent}. By Mackey's formula, $\chi|_C$ contains $\Ind^C_{J \cap C}(\psi|_{J \cap C})$. Hence it suffices to 
show that, for the irreducible character $\psi$ of $J$ that lies above $\lambda$, the restriction $\psi|_{J \cap C}$ contains 
$1_{J \cap C}$.

We will again use the map \eqref{com11} to identify $F$ with $V$. Fixing a primitive $p^{\mathrm {th}}$ root of unity $\eps \in \C^\times$,
we can find a unique element $0 \neq Z \in V$ such that 
\begin{equation}\label{f30}
  \lambda(T) = \eps^{\Tr(ZT)}
\end{equation}  
for any $T \in V$. We already noted that $G/\langle E,\bz \rangle$ acts on $V$ as $\Omega(V) \cong \Omega_3(p)$. Also, with respect to the trace form \eqref{tr1}, the plane $V_0$ is non-degenerate  and contains the isotropic vector $\begin{pmatrix}0 & 1\\0 & 0\end{pmatrix}$. Hence,
conjugating $\lambda$ suitably under $G$, by Lemma \ref{f3} we may assume that
\begin{equation}\label{f31}
  Z \in V_0.
\end{equation}    

As the trace form \eqref{tr1} is $G$-invariant, any element $x \in C$ fixes $\lambda$ if and only if it fixes $Z$. Note that any 
element in the diagonal subgroup $\langle \diag(t,t^{-1}) \rangle = \langle gE \rangle$ of $G_1$ that is neither $I_2$ nor $\bz=-I_2$ 
has only a one-dimensional fixed point subspace $\langle \diag(-1,1) \rangle$ in $V$. Hence any such element cannot fix $Z$ 
by the condition \eqref{f31}. In other words,
$$J \cap C = \langle A \rangle \times \langle \bz \rangle.$$
As $\chi$ is trivial at the central involution $\bz$, $\psi$ is also trivial at $\bz$. Therefore, it suffices to show that 
$\psi|_{\langle A \rangle}$ contains the principal character $1_{\langle A \rangle}$.

Let $\varphi$ be an irreducible constituent of $\psi|_E$. Since $\psi$ lies above the character $\lambda$ of $F \leq \ZB(E)$, $\varphi$ also
lies above $\lambda$, i.e. $\varphi|_F = \varphi(1)\lambda$.

First assume that $n=2$, so that $A$ has order $p$, $A \in F=E$, and $\varphi=\lambda$. Under the identification \eqref{com11}, $A$ corresponds to the element $\diag(-1,1)$ in $V$. The formula \eqref{f30} and the condition \eqref{f31} ensure that 
$\lambda(A) = 1 = \lambda(1)$, and so we are done.

Assume now that $n \geq 3$. The same argument as above shows that $A^{p^{n-2}}$ is contained in the kernel of $\lambda$, and so
also in the kernel of $\varphi$. On the other hand, consider any integer $0 \leq j \leq n-3$ and any integer $i$ coprime to $p$.
If $\Ker(\lambda)$ contains the entire $F_0$ (which corresponds to $V_0$ via \eqref{com11}), then \eqref{f30} shows
that $Z$ is a multiple of $\diag(-1,1)$, contrary to \eqref{f31}. Hence 
$$\Ker(\lambda) \not\geq F_0.$$
By Lemma \ref{p-com2}, there is some $B \in E$ such that 
$$[A^{ip^j},B] \in F \smallsetminus \Ker(\lambda).$$
It follows from Lemma \ref{p-com1} that $\varphi(A^{ip^j})=0$. 

We have shown that $\varphi$ vanishes at any element of the cyclic group $\langle A \rangle$ of order $p^{n-1}$ outside of the
subgroup of order $p$, on which $\varphi$ is trivial. Hence the multiplicity of $1_{\langle A\rangle}$ in $\varphi|_{\langle A\rangle}$ equals 
$$\frac{1}{p^{n-1}}\sum_{x \in \langle A \rangle}\varphi(x) = \frac{\varphi(1)}{p^{n-2}} > 0,$$
completing the proof.
\end{proof}

\begin{rmk}
\begin{enumerate}[\rm(i)]
\item It is possible that results of \cite{Ku} can also be used to obtain another proof of Theorem \ref{main-sl2}. An advantage of 
the current proof of Theorem \ref{main-sl2} is that it relies only on the structure of $\SL_2(\Z/p^n\Z)$ and does not require knowing its
character table.
\item Eamonn O'Brien has verified using {\sf Magma} that Theorem \ref{main-sl2} also hold for 
$\SL_2(\Z/2^n\Z)$ with $1 \leq n \leq 11$, and $\SL_2(\Z/3^n\Z)$ with $1 \leq n \leq 7$. In the case 
$p=2,3$, our proof of Theorem 
\ref{main-sl2} does show that many of the irreducible characters of 
$G/\ZB(G)$ occur in the conjugation character of $G = \SL_2(\Z/p^n\Z)$, but still leaves out some of these 
characters. It is not clear whether results of \cite{Ta} can be lead to a proof of Theorem \ref{main-sl2}
for $p=2,3$. 
\end{enumerate}
\end{rmk}

\end{document}